\documentclass[11pt,reqno]{amsart}
\usepackage{amssymb,latexsym,multirow,wasysym,wrapfig,amsthm}
\usepackage{graphicx}
\usepackage{enumitem}
\usepackage{color,amsmath,amssymb,amsthm,fullpage}
\usepackage{tikz}
\definecolor{lightgray}{gray}{0.90}
\usepackage{colortbl}
\usepackage{supertabular,multicol}
\newcount\n
\def\tablebody{}
\makeatletter
\loop\ifnum\n<100
        \advance\n by1
        \protected@edef\tablebody{\tablebody
                \textbf{\number\n.}& shortText
                \tabularnewline
        }
\repeat

\makeatletter
\let\mcnewpage=\newpage
\newcommand{\TrickSupertabularIntoMulticols}{%
  \renewcommand\newpage{%
    \if@firstcolumn
      \hrule width\linewidth height0pt
      \columnbreak
    \else
      \mcnewpage
    \fi
  }%
}
\makeatother

\usepackage{relsize}
\makeatletter
\newtheorem*{rep@theorem}{\rep@title}
\newcommand{\newreptheorem}[2]{%
\newenvironment{rep#1}[1]{%
 \def\rep@title{#2 \ref{##1}}%
 \begin{rep@theorem}}%
 {\end{rep@theorem}}}
\makeatother

\usepackage{url}
\usepackage{array}
\newcolumntype{P}[1]{>{\centering\arraybackslash}p{#1}}
\usepackage{supertabular}

\makeatletter
\newcommand{\addresseshere}{
  \enddoc@text\let\enddoc@text\relax
}
\makeatother
\newcommand{\Z}{{\mathbb Z}}
\newcommand{\Se}{S_{e,!}}
\newcommand{\Stwo}{S_{2,!}}

\theoremstyle{definition}
\newtheorem{definition}{Definition}

\newtheorem{theorem}{Theorem}
\newtheorem{lemma}{Lemma}

\newtheorem{prop}{Proposition}
\newtheorem{cor}{Corollary}

\newtheorem{problem}{Problem}

\newreptheorem{theorem}{Theorem}
\numberwithin{equation}{section}

\begin{document}

\title{Sequences of consecutive factoradic happy numbers}
 
 \author{Joshua Carlson}
 \address[J. Carlson]{Department of Mathematics and Statistics\\Williams College\\
33 Stetson Court\\
Williamstown, MA 01267, USA} 
 \email{jc331@williams.edu}
 
 \author{Eva G. Goedhart}
 \address[E. G. Goedhart]{Department of Mathematics and Statistics\\Williams College\\
33 Stetson Court\\
Williamstown, MA 01267, USA} 
 \email{eeg4@williams.edu}
 
 \author{Pamela E. Harris}
 \address[P. E. Harris]{Department of Mathematics and Statistics\\Williams College\\
33 Stetson Court\\
Williamstown, MA 01267, USA} 
 \email{peh2@williams.edu}

\begin{abstract}
Given a positive integer $n$, the factorial base representation of $n$ is given by
$n=\sum_{i=1}^ka_i\cdot i!$,
where $a_k\neq 0$ and $0\leq a_i\leq i$ for all $1\leq i\leq k$. 
For $e\geq 1$, we define $S_{e,!}:\Z_{\geq0}\to\Z_{\geq0}$ by $S_{e,!}(0) = 0$ and
$S_{e,!}(n)=\sum_{i=0}^{n}a_i^e$, for $n \neq 0$. 
For $\ell\geq 0$, we let $S_{e,!}^\ell(n)$ denote the $\ell$-th iteration of $S_{e,!}$, while $S_{e,!}^0(n)=n$.
If $p\in\Z^+$ satisfies $\Se(p)=p$, then we say that $p$ is an $e$-power factoradic fixed point of $S_{e,!}$. 
Moreover, given $x\in \mathbb{Z}^+$, if $p$ is an $e$-power factoradic fixed point and if there exists $\ell\in \Z_{\geq 0}$ such that  $\Se^\ell(x)=p$, then we say that $x$ is an $e$-power factoradic $p$-happy number. Note an integer $n$ is said to be an $e$-power factoradic happy number if it is an $e$-power factoradic $1$-happy number. 
In this paper, we prove that all positive integers are $1$-power factoradic happy and, for $2\leq e\leq 4$, we prove the existence of arbitrarily long sequences of $e$-power factoradic $p$-happy numbers. 
A curious result establishes that for any $e\geq 2$ the $e$-power factoradic fixed points of $S_{e,!}$ that are greater than $1$, always appear in sets of consecutive pairs. Our last contribution, provides the smallest sequences of $m$ consecutive $e$-power factoradic happy numbers for $2\leq e\leq 5$, for some values of $m$.
\end{abstract}
\maketitle
\section{Introduction}
Mixed radix numeral systems, studied by Cantor as early as the 1860's \cite{Cantor}, are non-positional number systems in which the weights associated with the positions do not form a geometric sequence. 
A common example of a mixed radix numeral system is the Gregorian calendar, the most widely used civil calendar, which counts years in decimal, but months in duodecimal \cite{Richards}. This is also true of the Mayan's numeral system, a generalization of base 20, since some positions represented a multiplication by 18 rather than 20, so that the calendar year would correspond to 360 days, very closely approximating the solar year \cite{Addie}.

In this paper we consider a mixed radix number system known as the factoradic number system, also called the factorial number system. 
Given a positive integer $n$, if $k\in \Z^+$ is the largest integer satisfying  
$k!\leq n<(k+1)!$,
then the factoradic representation of $n$ is given by
\[n=\sum_{i=1}^ka_i\cdot i!\]
where $0\leq a_i\leq i$ for all $1\leq i\leq k$.
The following procedure, due to Laisant \cite{Laisant}, finds the digits of the factorial representation a positive integer.
Given $n\in\Z^+$, successively divide $n$ by the radix $i$, beginning with $i=2$, and take the remainder of the result as the digit $a_{i-1}$. Note that we begin with $i=2$, as dividing by $1$, would always yield $a_0=0$, and we omit this digit as it has no effect on the representation of any positive integer.
Since $n$ is finite, continuing this process using the resulting integer quotient, this quotient eventually becomes 0. 
Thus, this process terminates and provides the full set of digits $a_i$ for $1\leq i\leq k$ in the factorial representation of $n$. For example, to find the factorial representation of 2020 we compute
\begin{center}
    \begin{tabular}{rccclcc}
    2020&$\div$& 2&=&1010&remainder &$0$,\\
    1010&$\div$& 3&=&336&remainder &$2$,\\
    336&$\div$& 4&=&84&remainder &$0$,\\
    84&$\div$& 5&=&16&remainder &$4$,\\
    16&$\div$& 6&=&2&remainder &$4$,\\
    2&$\div$& 7&=&0&remainder &$2$,
\end{tabular}
\end{center}
and so $2020=2\cdot 6!+4\cdot 5!+4\cdot 4!+0\cdot3!+2\cdot2!+0\cdot1!$.
Lastly, note that general properties of positional number systems also apply to the factorial number system. 
For instance, the factoradic representation of any positive number is unique since the sum of the respective factorials multiplied by the index is always the next factorial minus one, namely,
$\sum _{i=1}^{k}i\cdot i!=(k+1)!-1$. 
For a direct proof of the uniqueness of factorial representations see \cite{unambiguity}.

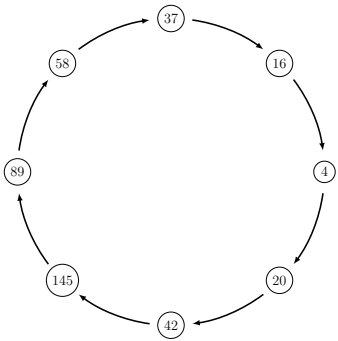
\begin{wrapfigure}{l}{6.9cm}
    \centering
    \resizebox{4.5cm}{!}{
    \begin{tikzpicture}
 \def \n {8}
\def \radius {5.5cm}
\def \margin {8} % margin in angles, depends on the radius
\node[draw, circle] at ({360/8 * (1 - 1)}:\radius) {\Large$4$};
\node[draw, circle] at ({360/8 * (2 - 1)}:\radius) {\Large$16$};
\node[draw, circle] at ({360/8 * (3 - 1)}:\radius) {\Large$37$};
\node[draw, circle] at ({360/8 * (4 - 1)}:\radius) {\Large$58$};
\node[draw, circle] at ({360/8 * (5- 1)}:\radius) {\Large$89$};
\node[draw, circle] at ({360/8 * (6 - 1)}:\radius) {\Large$145$};
\node[draw, circle] at ({360/8 * (7 - 1)}:\radius) {\Large$42$};
\node[draw, circle] at ({360/8 * (8 - 1)}:\radius) {\Large$20$};
  \foreach \s in {1,...,\n} 
{
  \draw[<-, ultra thick,>=latex] ({360/\n * (\s - 1)+\margin}:\radius) 
    arc ({360/\n * (\s - 1)+\margin}:{360/\n * (\s)-\margin}:\radius);
}
    \end{tikzpicture}
    }
\vspace*{7mm}
    \caption{Base $10$ happy function cycle.}
    \label{fig:cycle}
\end{wrapfigure}
Our work extends the definition of a happy number to the factorial representation of positive integers. We begin by recalling that the happy function  $S_2:\Z^+\to\Z^+$ takes a positive integer to the sum of the square of its decimal digits. Namely, if $n=\sum_{i=0}^{k}a_i\cdot 10^i$, with $0\leq a_i\leq 9$ for all $1\leq i\leq k$, then 
\begin{align*}
    \hspace{3in}S_e(n)&=\sum_{i=0}^ka_i^2.
\end{align*}
The positive integer $n$ is said to be happy if repeated iteration of the happy function results in $1$. 
Honsberger \cite{Honsberger} established that if $n\in\Z^+$, then $n$ is happy or after repeated iteration of the happy function the result lies in the cycle illustrated in Figure \ref{fig:cycle}.
Since this initial analysis, many have studied the behavior of consecutive (or $d$-consecutive) happy numbers in positive, negative, and fractional bases, as well as with new positive exponents defining generalized happy functions \cite{siksek,GH18,GT07,GT01,Pan,Styer,Trevino}.
To begin, we adapt the definition of the happy function and happy numbers, as given in~\cite{GT01}, to the factorial representation of an integer. 
\begin{definition}
Let $e\geq 1$ be an integer, and let $n\in \Z^+$ have factoradic representation $n=\sum_{i=1}^ka_i\cdot i!$ for $k\in \Z^+$ where $a_k\neq 0$ and $0\leq a_i\leq i$ for each $1\leq i\leq k$. Then define the $e$-power factoradic happy function $\Se:\Z_{\geq0}\to\Z_{\geq0}$ by 
\[\Se(n)=\begin{cases}
 0, & \text{if $n=0$}\\[.5 em]
\displaystyle \sum_{i=1}^{k}a_i^e, & \text{if $n\geq 1$.}
\end{cases}\]
For each $\ell\in\Z^+$ and $n\in \Z_{\geq 0}$, we denote the $\ell$-th iteration of the function $\Se(n)$ by $\Se^{\ell}(n) = \Se(\Se^{\ell-1}(n))$, while $\Se^0(n)=n$.
\end{definition}

\begin{definition}\label{def:e-power}
An integer $n$ is an \emph{$e$-power factoradic happy number} if, for some $\ell \in \Z^+$, $\Se^{\ell}(n) = 1$.  
\end{definition}

In what follows, we refer to $2$-power factoradic happy numbers as \emph{factoradic happy numbers}. As an example, $2020$ is a factoradic happy number since $S_{2,!}^5(2020)=1$, whereas $4$  is not because $S_{2,!}^{\ell}(4)=4$ for all $\ell\in\Z_{\geq 0}$.
In fact, if $p\in\Z^+$ satisfies $\Se(p)=p$, then we say that $p$ is an \emph{$e$-power factoradic fixed point}. 
Moreover, if there exists $\ell\in \Z_{\geq 0}$ such that  $\Se^\ell(x)=p$, we say that $x$ is an \emph{$e$-power factoradic $p$-happy number}. For example, $2021$ is a $2$-power factoradic $5$-happy number, since $S_{2,!}(5)=5$ and $S_{2,!}^3(2021)=5$.

In the happy number literature it is standard to ask about the existence of arbitrarily long sequences of consecutive happy numbers. 
Our main result generalizes the work of El-Sedy and Siksek~\cite{siksek}
and Grundman and Teeple~\cite{GT07} 
to establish a more general fact for the factoradic base system. 

\begin{theorem}\label{thm:all e}
Let $e,p \in \mathbb{Z}^+$ and $U_{e,!}$ is absolute and $(e,p)$-nice (Definition \ref{def:absolute} and \ref{def:nice}, respectively). Then there exists an arbitrarily long sequence of consecutive integers which iterate to $p$ after repeated application of $\Se$. 
\end{theorem}

Unfortunately, the sufficient conditions of Theorem \ref{thm:all e} are computationally difficult to satisfy. For $e\in \{1,2,3,4\}$, we can prove the following.

\begin{theorem}
\label{thm:main1}
For $e\in\{1,2,3,4\}$ and for any $e$-power factoradic fixed point $p$ of $\Se$, there exists arbitrarily long sequences of $e$-power factoradic $p$-happy numbers.
\end{theorem}

Note that Theorem \ref{thm:main1} implies that there exists arbitrarily long sequences of $e$-power factoradic happy numbers, whenever $e\in\{1,2,3,4\}$.

The manuscript is organized as follows. In Section \ref{sec:fixed points and cycles}, we determine the $e$-power factoradic fixed points and cycles of the functions $S_{e,!}$ for $1 \leq e \leq 6$.  In Section~\ref{sequencesection}, we prove Theorem \ref{thm:all e} and \ref{thm:main1}. In Section \ref{sec:small seqs}, we provide the smallest sequences of $m$ consecutive $e$-power factoradic happy numbers for $2\leq e\leq 5$, for some values of $m$.
We end with Section \ref{sec:future} in which we provide a direction for future research.

\section{Fixed points and cycles of factoradic happy functions}
\label{sec:fixed points and cycles}
In this section, we investigate the $e$-power factoradic fixed points and cycles of the function $S_{e,!}$. 
However, we begin by making a few observations. For a positive integer represented in factorial base, we have $n=\sum_{i=1}^ka_i \cdot i!$ where $a_1$ is always $1$ or $0$ depending on the parity of $n$.
This means that $a_1-a_1^{e}=0$ for every $a_1\in\{0,1\}$ and $e\in\Z^+$. 
Hence, for $2\leq i\leq k$, if $n$ is odd and subsequently, $a_1=1$, $n-1=\sum_{i=2}^ka_i\cdot i!$. 
Whereas, if $n$ is even and $a_1=0$, we have that $n+1=1+\sum_{i=2}^ka_i\cdot i!$ with all the same coefficients $a_i$ for $2\leq i\leq k$. The following curious result follows easily from this observation. 

\begin{lemma}\label{lem:alles}Let $n$ and $e$ be positive integers.
\begin{enumerate}[label=(\roman*)]
    \item If $n$ is odd, then $n-\Se(n)=(n-1)-\Se(n-1)$.\label{case:1}
    \item If $n$ is even, then $n-\Se(n)=(n+1)-\Se(n+1)$.\label{case:2}
\end{enumerate}
\end{lemma}
\begin{proof}
Let $n$ and $e$ be given positive integers. Then $n=\sum_{i=1}^ka_i \cdot i!$, with $a_k\neq 0$ and $0\leq a_i\leq i$ for all $1\leq i\leq k$. Hence 
\begin{align}
    n-\Se(n)&=\sum_{i=1}^ka_i \cdot i!-\sum_{i=1}^ka_i ^e\nonumber\\
    &=a_1-a_1^{e}+\sum_{i=2}^k\left(a_i\cdot i!-a_i^{e}\right)\nonumber\\
    &=\sum_{i=2}^k(a_i\cdot i!-a_i^{e})\label{eq:1}
\end{align}
where the last equality holds since $a_1-a_1^{e}=0$ for all $n$, as observed.

Now, if we further assume that $n$ is odd, then $a_1=1$. By the observation above, notice that $n-1=\sum_{i=2}^ka_i \cdot i!+ 0\cdot1!$, with $a_i$ as in the factoradic expansion of $n$.
Hence, by \eqref{eq:1}, 
\begin{align}
    (n-1)-\Se(n-1)
    &=\sum_{i=2}^k(a_i\cdot i!-a_i^{e})\label{eq:2}=n-\Se(n).
\end{align}
The proof of  \ref{case:2} follows analogously from the observation.
\end{proof}

Next, we exhibit a peculiar property of the $e$-power factoradic fixed points of $S_{e,!}$ for any positive integer~$e$. 

\begin{prop}\label{prop:fixed points come in pairs}
For each integer $e \geq 1$, the subset of $e$-power factoradic fixed points that are greater than one, $F_e = \{n \in \mathbb{Z}^+ \ | \ n > 1 \text{ and } S_{e,!}(n) = n\},$ consists of pairs of consecutive integers. 
\end{prop}

\begin{proof}
Let $e\geq 1$ be an $e$-power factoradic fixed integer and let $n \in F_e$. 
Since $n > 1$, $n = \sum_{i=1}^k a_i\cdot i!$ where $k \geq 2$, $a_k \neq 0$, and $0 \leq a_i \leq i$ for each $1\leq i\leq k$. Also, since $n\in F_e$, $S_{e, !}(n) = n$. It follows that 
\begin{align*}
0=n - S_{e,!}(n)
= \sum_{i=2}^k(a_i\cdot i!-a_i^{e}).
\end{align*}
Now, applying Lemma~\ref{lem:alles}, if $n$ is odd, we have that $(n-1) - S_{e,!}(n-1)=n - S_{e,!}(n)=0$ and so $n-1\in F_e$. Similarly, if $n$ is even, apply Lemma~\ref{lem:alles} to get $n+1\in F_e$.
\end{proof}

The remainder of this section is dedicated to determining the exact set of $e$-power factoradic fixed points of $S_{e,!}$ for specific values of $e$, see Table~\ref{tab:fixed points and cycles} below. First, consider the case where $e = 1$. 
\begin{prop}\label{prope1}
If $n > 1$, then $S_{1,!}(n) < n$.
\end{prop}
\begin{proof}
If $n > 1$, then $k > 1$ where $n = \sum_{i=1}^k a_i \cdot i!$ with $ a_k \neq 0$ and $0 \leq a_i \leq i$ for each $1\leq i \leq k$. Since $k > 1$ and $a_k > 0$, we have that \[n - S_{1,!}(n)
= \sum_{i=2}^k (a_i\cdot i! - a_i) \geq a_k\cdot k! - a_k > 0.\] Thus, $S_{1,!}(n) < n$, whenever $n > 1$.
\end{proof}

Since $S_{1,!}(1)=1$, Proposition \ref{prope1} implies that $n=1$ is the only $e$-power factoradic fixed point of $S_{1,!}$. Additionally, by Definition~\ref{def:e-power} and Proposition \ref{prope1} 
we also obtain the following corollary.

\begin{cor}\label{cor:e=1}
Every positive integer is a $1$-power factoradic happy number.
\end{cor}

Next, for each $e$ in $2\leq e\leq 4$, we establish a point for which the iterations of the happy function decrease. 
We restrict to $2\leq e\leq 4$ as our method requires a closed formula for the gamma function, $\Gamma(x)$, a continuous extension of the discrete factorial function, in order to find the smallest positive integer solution to the equation $x!>x^{e-1}$ for $x\geq 0$. 
Finding a solution to $\frac{d}{dx}(\Gamma(x+1) - x^{e-1}) = 0$ is beyond the scope of this paper as it is known that $\Gamma(x)$ does not have a general closed formula containing only elementary functions \cite{GR}.

To continue, we require the following technical result.
\begin{lemma}\label{lem:e=2 to 4}
Let $e\in\{2,3,4\}$ be fixed. If $j_e\in\Z^+$ is the smallest integer such that $j_e!>j_e^{e-1}$, then for all integers $k\geq j_e$, we have that $k!>k^{e-1}$ and $(k+1)!-(k+1)^{e-1}\geq k!-k^{e-1}$.
\end{lemma}
\begin{proof}
We reduce to considering cases for each possible $e$ and proceed by induction. Since the argument for each case is very similar and for larger $e$ the bounds gets a little tighter, we provide the justification for the case when $e=4$ only. 

Let $e=4$, then $j_4=6$.
For induction, we assume that $k!>k^3$ for $6\leq k$. Now,
\begin{align*}
   (k+1)^3&=k^3+3k^2+3k+1
    <k^3+k^3+k^3+k^3
    =4k^3
    <4\cdot k! 
    <(k+1)\cdot k!
    =(k+1)!.
\end{align*}
Next, $k^3>(k+1)^2$ for $k\geq 6$ implies that
\begin{align*}
    (k+1)!-(k+1)^3&= (k+1)[k!-(k+1)^2]>(k+1)(k!-k^3)>k!-k^2. \qedhere
\end{align*}
\end{proof}

\begin{theorem}\label{thm:corrected Me result}
Let $e\in\{2,3,4\}$ be given and let $j_e\in\Z^+$ be the smallest integer so that  $j_e!>j_e^{e-1}$. If $n\in\Z^+$ with $n>M_e$ where $M_e=\sum_{i=1}^{j_e}i\cdot i!$, then $n-\Se(n)>0$.
\end{theorem}
\begin{proof}
As in the previous lemma, we proceed via a case-by-case analysis for each possible value of $e$. The proofs of each case are analogous to the following.\\

Let  $e=4$, then $j_4=6$ and $M_4=5039$. Let $n\in\Z^+$ and assume that $n=\sum_{i=1}^{k}a_i\cdot i!>M_4$. Then, by the definition of $M_4$, it must be that $k>6$. Now notice 
\begin{align*}
    n-\Se(n)&=
    \sum_{i=1}^{k}a_i( i!-a_i^3)=\sum_{i=6}^{k}a_i( i!-a_i^3)+\sum_{i=1}^{5}a_i( i!-a_i^3).
\end{align*}
Recall that $a_1(1!-a_1^3)=0$ for $0\leq a_1\leq 1$. 
Next, for each $1\leq i\leq 5$, we have the digits $a_i$ are bounded as in $0\leq a_i\leq i$. 
Thus, we can find the minimum possible value of $a_i(i!-a_i^3)$ for $1\leq i\leq 5$, resulting in $\sum_{i=1}^{5}a_i(i!-a_i^3)\geq -260$. 
Also by Lemma \ref{lem:e=2 to 4}, we know $i!-a_i^3>0$ and, hence, that $a_i( i!-a_i^3)\geq 0$ for all $6=j_e\leq i\leq k$. 
Hence, 
\begin{align*}
    n-\Se(n)&\geq\sum_{i=6}^{k}a_i( i!-a_i^3)-260\\
    &\geq a_k(k!-a_k^3)-260\\
    &>k!-k^3-260\\
    &>7!-7^3-260 >0.
    \qedhere
\end{align*}
\end{proof}

Now, in order to understand the iterative behavior of the function $\Se$, Theorem \ref{thm:corrected Me result} implies that it suffices to compute iterates of $\Se$ on $n$ for $1\leq n\leq M_e$. Using this result, we can find the set of fixed points of $\Se$.

\begin{definition}\label{def:absolute}
Define $U_{e,!}$ to be the set of all positive integers $n$ for which there exists $\ell\geq 1$ satisfying $\Se^\ell(n)=n$. Further, we say the set $U_{e,!}$ is \emph{absolute} if for all $n\in\Z^+$, there exists a nonnegative integer $r_n$ such that $S_{e,!}^r(n) \in U_{e,!}$ for all integers $r \geq r_n$.
\end{definition}
Note that in the case that $\ell=1$, $n$ is an $e$-power factoradic fixed point of $\Se$, and whenever $\ell>1$, then $n$ lies in a cycle of $\Se$. Thus, $U_{e,!}$ is the set of all integers that are fixed points or appear in a cycle of $\Se$. 

\begin{prop}\label{prop:UEabsolute}
For $e \in \{1,2,3,4\}$, $U_{e,!}$ is absolute.
\end{prop}

\begin{proof}
The fact that $U_{1,!}$ is absolute follows from Corollary~\ref{cor:e=1}. Now, fix $e \in \{2,3,4\}$. First, assume that $n\in\Z^+$ such that $n>M_e$. In this case, Theorem~\ref{thm:corrected Me result}, implies that $n > S_{e,!}(n)$. So for each $n\in\Z^+$, if there exists a nonnegative integer $\ell$ such that $S_{e,!}^{\ell}(n) > M_e$, then there exists an integer $r > \ell$ such that $S_{e,!}^r(n) \leq M_e$. 

Now, assume $n\in \Z^+$ such that $1\leq n\leq  M_e$ and  $S_{e,!}^r(n) \leq M_e$ for some integer $r \geq 0$. By the pigeonhole principle, there exists distinct integers $p,q \geq 0$ such that $S_{e,!}^p(n) = S_{e,!}^q(n)$. This means that for some  large enough $\ell\geq 0$, $S_{e,!}^\ell(n)$ is an $e$-power factoradic fixed point or lies in a cycle of $\Se$. Hence, there exists a nonnegative integer $r_n$ such that $S_{e,!}^r(n) \in U_{e,!}$ for all integers $r \geq r_n$. Therefore, $U_{e,!}$ is absolute for each $e \in  \{1,2,3,4\}$.
\end{proof}

The importance of  Theorem \ref{thm:corrected Me result} and Proposition \ref{prop:UEabsolute} is in having the integer $M_e$. It allows us to compute all $e$-power factoradic fixed points and cycles of the functions $\Se$ by checking the iterations of the function $\Se$ for integers in the interval $[1,M_e]$. 
This computational work is summarized in Table \ref{tab:fixed points and cycles} when $1\leq e\leq 6$.

\begin{table}[h!]
    \centering
    \begin{tabular}{|l|l|l|l|l|}\hline\rowcolor{lightgray}
         $e$&$M_e$& $e$-power factoradic fixed points &Cycles\\\hline\hline
         1&5& 1&None \\\hline
         2&23& 1, 4, 5& None \\\hline
         3&119&1, 16, 17& None \\\hline
         4&5039&1, 658, 659& None\\\hline
         5&40319&1, 34, 35, 308, 309, 1058, 1059& (3401,2114)\\\hline
         6&362879&1, 8258, 8259& (731, 67, 794)\\\hline
    \end{tabular}
    \vspace{2mm}
    \caption{Base 10 representation of $e$-power factoradic fixed points and cycles of $S_{e,!}$ for $1\leq e\leq 6$.}
    \label{tab:fixed points and cycles}
\end{table}

{\bf Remark.} The proof of Theorem~\ref{thm:corrected Me result} can be extended to $e\geq 5$ using similar techniques. However, in these cases we begin to see cycles and, hence, proving that $U_{e,!}$ is $(e,p)$-nice (Definition \ref{def:nice}) becomes very difficult as $U_{e,!}$ is much larger. This demonstrates the necessity for new techniques. 

\section{Consecutive sequences of factoradic happy numbers}\label{sequencesection}
In this section, we prove the existence of sequences of consecutive $e$-power factoradic happy numbers for various values of $e$. We also obtain similar results about sequences of integers that eventually iterate to integers other than $1$. In \cite[Theorem 3.1]{siksek}, a technique is given to find sequences of happy numbers of arbitrary length. We start by generalizing this technique in order to apply it to factoradic happy numbers. We begin by stating some important definitions and establishing some needed technical results.

\begin{definition}
Let $n\geq 1$ be an integer with factorial base representation $n=\sum_{i=1}^ka_i\cdot i!$ with $0\leq a_i\leq i$ for all $0\leq i\leq k$ and $a_k\neq 0$. For each integer $t\geq 0$, define $f_t(n)=\sum_{i=1}^{k}a_i\cdot (t + i)!$.
\end{definition}
Note that for any integers $n \geq 1$ and $t \geq 0$, we have that $f_t(n)$  is the number whose factorial representation is the same as that of $n$, but with $t$ many zeros appended at the end. For example, for $5=2\cdot 2!+1\cdot 1!$, we have $f_2(5)=2\cdot 4!+1\cdot 3!+0\cdot 2!+0\cdot 1!=54$. The next definition is a relaxed version of the factorial analog to the definition of $(e,b)$-good in \cite{GT07}.

\begin{definition}\label{def:nice}
For integers $e,p \in \mathbb{Z}^+$, a set $D$ is \emph{$(e,p)$-nice} if there exists an integer $l\geq 0$ such that for all $d \in D$, there exists an integer $q_d \geq 0$ such that $S_{e,!}^{q_d}(l + d) = p$.
\end{definition}
In other words, a set $D$ is $(e,p)$-nice if we can find a nonnegative integer $l$ to create a new sequence $D_{l}=\{d+l: d\in D\}$, which has the property  $\Se^{q_d}(d+l)=p$ for some $q_d$ depending on $d$.

The following two results concern the preimage of $S_{e,!}$ and the relationship between the functions $S_{e,!}$ and $f_t$.
\begin{lemma}\label{preimage}
For each $e,x \in \mathbb{Z}^+$, the preimage $S_{e,!}^{-1}(x)=\{y\in\Z^+: S_{e,!}(y)=x\}$ is nonempty.
\end{lemma}
\begin{proof}
Let $e, x\in \Z^+$, then $S_{e,!}(\sum_{i=1}^x1\cdot i!)=\underbrace{1^e+\cdots+1^e}_{x}=x$.
\end{proof}

\begin{lemma}\label{ftse}
Let $e\geq 1$ be fixed. If $x,y,t\in\Z^+$ with $t$ at least the number of digits in the factorial representation of $y$, then \[S_{e,!}(f_t(x)+y)=S_{e,!}(x)+S_{e,!}(y).\]
\end{lemma}
\begin{proof}
For each $x\in \Z^+$, let $E_x$ be the multiset of nonzero digits in the factorial representation of~$x$. By the condition given on $t$, we know that $E_{f_t(x)+y}$ is the disjoint union of $E_x$ and $E_y$.
\end{proof}

We are now ready to use a factorial analog of the proof of \cite[Theorem 3.1]{siksek} to obtain a slightly more general result.

\begin{reptheorem}{thm:all e}
Let $e,p \in \mathbb{Z}^+$ and $U_{e,!}$ is absolute and $(e,p)$-nice. Then there exists an arbitrarily long sequence of consecutive integers which iterate to $p$ after repeated application of $\Se$. 
\end{reptheorem}

\begin{proof}
To prove this result, it is sufficient to show that if  $e,p \in \mathbb{Z}^+$ and $U_{e,!}$ is absolute and $(e,p)$-nice, then for any  $m \in \mathbb{Z}^+$, there exists a sequence of consecutive integers $n_1, n_2, \ldots ,n_m$ such that for each $i \in \{1, \ldots, m\}$, $S_{e,!}^{n_i'}(n_i) = p$ for some non-negative integer $n_i'$. 

Fix $e,p \in \mathbb{Z}^+$ and suppose $U_{e,!}$ is absolute and $(e,p)$-nice. Let $m$ be an arbitrary positive integer. We will construct a sequence of $m$ consecutive integers that eventually iterate to $p$. Since $U_{e,!}$ is absolute, there exists an integer $r \geq 1$ such that $S_{e,!}^r(i) \in U_{e,!}$ for each $i \in \{1, 2, \ldots, m\}$. Define $t$ to be the maximum number of digits in the factorial representation of $S_{e,!}^j(i)$ taken over all $0 \leq j \leq r$ and $1 \leq i \leq m$. Since $U_{e,!}$ is $(e,p)$-nice, there exists an integer $l \geq 0$ such that for all $u \in U_{e,!}$, there exists an integer $q_u \geq 0$ such that $S_{e,!}^{q_u}(l + u) = p$. Define $l_r = l$ and for each $1 \leq j \leq r-1$, define $l_j$ from $l_{j+1}$ as follows. By Lemma \ref{preimage}, there exists a positive integer $k_{j+1}$ such that $S_{e,!}(k_{j+1}) = l_{j+1}$. Define $l_j$ to be $f_t(k_{j+1})$.

Now, by the choice of $t$ and the definition of $f_t$, we have that for each $0 \leq j \leq r-1$, \[S_{e,!}(l_j) = S_{e,!}(f_t(k_{j+1})) = S_{e,!}(k_{j+1}) = l_{j+1}.\] Also, it follows from Lemma \ref{ftse} that for any $y \in \mathbb{Z}^+$, if $0 \leq j \leq r-1$ and the number of digits in the factorial representation of $y$ is at most $t$, then \[S_{e,!}(l_j + y) = S_{e,!}(f_t(k_{j+1}) + y) = S_{2,!}(k_{j+1}) + S_{2,!}(y) = l_{j+1} + S_{e,!}(y).\] Therefore, for each $i \in \{1, 2, \ldots, m\}$, 
\begin{eqnarray*}
S_{e,!}^r(l_0 + i) &=& S_{e,!}^{r-1}(S_{e,!}(l_0 + i)) = S_{e,!}^{r-1}(l_1 + S_{e,!}(i))\\
&=& S_{e,!}^{r-2}(S_{e,!}(l_1 + S_{e,!}(i))) = S_{e,!}^{r-2}(l_2 + S_{e,!}^2(i))\\
&\vdots& \\
&=& l_r + S_{e,!}^r(i) = l + S_{e,!}^r(i).
\end{eqnarray*}
Note that $r$ was chosen such that for each integer $1 \leq i \leq m$, $S_{e,!}^r(i) \in U_{e,!}$. By the definition of $(e,p)$-nice, we have that for each $i \in \{1, 2, \ldots, m\}$, there exists an integer $q_i \geq 0$ such that $S_{e,!}^{q_i}(l + S_{e,!}^r(i)) = p$. Thus, for each $i \in \{1, 2, \ldots, m\}$, if $n_i = l_0 + i$ and $n_i' = q_i + r$, then \[S_{e,!}^{n_i'}(n_i) = S_{e,!}^{q_i + r}(n_i) = S_{e,!}^{q_i + r}(l_0 + i) = S_{e,!}^{q_i}(S_{e,!}^r(l_0 + i)) = S_{e,!}^{q_i}(l + S_{e,!}^r(i)) = p.\qedhere \]
\end{proof}

\begin{reptheorem}{thm:main1}
For $e\in\{1,2,3,4\}$ and for any $e$-power factoradic fixed point $p$ of $\Se$, there exists arbitrarily long sequences of $e$-power factoradic $p$-happy numbers.
\end{reptheorem}
\begin{proof}
By Theorem \ref{thm:all e} and Proposition \ref{prop:UEabsolute}, it suffices to show that $U_{e,!}$ is $(e,p)$-nice for $p$ an $e$-power factoradic fixed point of $\Se$ and $e\in\{1,2,3,4\}$. 
As every positive integer is a $1$-power factoradic happy number, we start our investigation with $e=2$.

\noindent\textbf{Case 1:} 
If $p$ is a $2$-power factoradic fixed point of $S_{2,!}$, then $p\in\{1,4,5\}$.
First, we show that $U_{2,!}$ is $(2, 1)$-nice. If $l = 20$, then $U_{2,!} + 20 = \{1, 4, 5\} + 20 = \{21, 24, 25\}$.
The result follows from the fact that $\Stwo^3(21)$, $\Stwo(24)$, and $\Stwo^2(25)$ all equal $1$.

To show that $U_{2,!}$ is $(2,4 )$-nice, choose $l = 2841$. Then, $U_{2,!} + 2841 = \{1, 4, 5\} + 2841 = \{2842, 2845, 2845\}$.
The result follows from the fact that $\Stwo^2(2842)$, $\Stwo(2845)$, and $\Stwo(2846)$ are equal to 4.

To show that $U_{2,!}$ is $(2,5)$-nice, choose $l = 45$. Then, $U_{2,!} + 45 = \{1, 4, 5\} + 45 = \{46, 49, 50\}$.
The result follows from the fact that $\Stwo^2(46)$, $\Stwo(49)$, and $\Stwo(50)$ are all equal to $5$.

\noindent\textbf{Case 2:} 
If $p$ is a $3$-power factoradic fixed point of $S_{3,!}$, then $p\in\{1,16,17\}$.
First, we show that $U_{3,!}$ is $(3, 1)$-nice. If $l = 2$, then $U_{3,!} + 2 = \{1, 16, 17\} + 2 = \{3, 18, 19\}$.
The result follows from the fact that $S_{3,!}^2(3)$, $S_{3,!}^4(18)$, and $S_{3,!}^5(19)$ are all equal to $1$.

Next, we show that $U_{3,!}$ is $(3, 16)$-nice. If $l = 50127$, then $U_{3,!} + 50127 = \{1, 16, 17\} + 50127 = \{50128, 50143, 50144\}$.
The result follows from the fact that $S_{3,!}^2(50128)$, $S_{3,!}^3(50143)$, $S_{3,!}^3(50144)$ are all equal to $16$.

Lastly, we show that $U_{3,!}$ is $(3, 17)$-nice. If $l = 4506$, then $U_{3,!} + 4506 = \{1, 16, 17\} + 4506 = \{4507, 4522, 4523\}$.
The result follows from the fact that $S_{3,!}^2(4507)$, $S_{3,!}^2(4522)$, and $S_{3,!}^2(4523)$ are all equal to $17$.

\noindent\textbf{Case 3:} If $p$ is a $4$-power factoradic fixed point of $S_{4,!}$, then $p\in\{1,658,659\}$.
First, we show that $U_{4,!}$ is $(4, 1)$-nice. If $l = 6$, then $U_{4,!} + 6 = \{1, 658, 659\} + 6 = \{7, 664, 665\}$.
The result follows from the fact that $S_{4,!}^2(7)$, $S_{4,!}^{12}(664)$, $S_{4,!}^{12}(665)$ are all equal to 1.

Next, we show that $U_{4,!}$ is $(4, 168)$-nice. If $l = 65763$, then $U_{4,!} + 65763 = \{1, 658, 659\} + 65763 = \{65764, 66421, 66422\}$.
The result follows from the fact that $S_{4,!}(65764)$, $S_{4,!}^2(66421)$, and $S_{4,!}^2(66422)$ are all equal to $658$.

Lastly, we show that $U_{4,!}$ is $(4, 169)$-nice. If $l = 31743$, then $U_{4,!} + 31743 = \{1, 658, 659\} + 31743 = \{31744, 32401, 32402\}$.
The result follows from the fact that $S_{4,!}^2(31744)$, $S_{4,!}^2(32401)$, and $S_{4,!}^2(32402)$ are all equal to $659$.
\end{proof}

\section{Smallest strings of consecutive $e$-power factoradic happy numbers}\label{sec:small seqs}
In 2007, Grundman and Teeple gave a list of the least examples of sequences of happy numbers of length 1-5 \cite{GT07}. This was extended by Styer in 2009, who found the smallest string of consecutive happy numbers of length 6-13 and also the smallest sequence of $3$-consecutive cubic happy numbers of lengths 4-9 \cite{Styer}.
In Table \ref{tab:small}, we provide the smallest sequences of $m$ consecutive $e$-power factoradic happy numbers for $2\leq e\leq 5$, for some values of $m$.
\begin{table}[h]
    \centering
    \begin{tabular}{|c|l|l|}\hline\rowcolor{lightgray}
        $e$ & $m$&  
         Sequences
        \\\hline\hline
         \multirow{3}{*}{2}&$m=1,2$&$2,1+m$\\\cline{2-3}
         & $m=3,4$&$6,7,\ldots,5+m$\\\cline{2-3}
         & $m=5,6,7,\ldots,11 $&$112,113,114,\ldots,111+m$\\\hline
         \multirow{3}{*}{3}& $m=1,2,3,\ldots, 13$&$2,3,4,\ldots,
         1+m$ \\\cline{2-3}
         & $m=14,15,16,\ldots, 21$&$18,19,20,\ldots,17+m$ \\\cline{2-3}
         & $m=22,23,24,\ldots, 31$&$63,64,65,\ldots,62+m$ \\\cline{2-3}
         & $m=32,33,34,\ldots, 41$&$95,96,97,\ldots,94+m$ \\
         \hline
         \multirow{1}{*}{4}&$m=1,2,\ldots, 602$ & $2,3,4, \ldots,1+m$\\
         \hline
         \multirow{1}{*}{5}&$m=1,2,3,\ldots,10$ & $2,3,4,\ldots,1+m$\\\hline
         \end{tabular}
         \vspace{2mm}
    \caption{Smallest string of $m$ consecutive $e$-power factoradic happy numbers when $2\leq e\leq 5$ and various values of $m$.}
    \label{tab:small}
\end{table}

\section{Future work}\label{sec:future}

In \cite{Gilmer}, Gilmer computed bounds for the proportion of happy numbers (in base $b\geq 2$), establishing that the upper density, $\bar{d}$, and the lower density, $\underline{d}$, of happy numbers satisfy 
\[\bar{d} > 0.18577 \mbox{ and } \underline{d} < 0.1138.\] Gilmer also proved that the 
asymptotic density does not exist for several generalizations of happy numbers. Thus, results finding bounds on the density of happy numbers are the best possible. 

\begin{table}[h]
    \centering
    \begin{tabular}{|c|p{1.9in}p{1.9in}p{1.9in}|}\hline\rowcolor{lightgray}
        $e$ &\multicolumn{3}{|c|}{Proportions of $e$-power factoradic fixed points of $\Se$ in the interval $I=[1,10!]$}\\[5
        pt]\hline\hline
        2         &$P_{2,1}(I)=\frac{2220945}{10!}=0.612$,& $P_{2,4}(I)=\frac{244026}{10!}=0.067$,& $P_{2,5}(I)=\frac{1163828}{10!}=0.321$\\[5pt]\hline
        3         &$P_{3,1}(I)=\frac{3421678}{10!}=0.943$,&  $P_{3,16}(I)=\frac{31856}{10!}=0.009$,& $P_{3,17}(I)=\frac{175265}{10!}=0.048$\\[5pt]
        \hline
        4         &$P_{4,1}(I)=\frac{3556797}{10!}=0.980$,& $P_{4,658}(I)=\frac{29574}{10!}=0.008$,& $P_{4,659}(I)=\frac{42428}{10!}=0.012$\\[5pt]\hline
         \multirow{3}{*}{5}
        &
        $P_{5,1}(I)=\frac{179930}{10!}=0.049$,& $P_{5,34}(I)=\frac{1545589}{10!}=0.426$,& $P_{5,35}(I)=\frac{38188}{10!}=0.0105$,\\[5pt]
        &$P_{5,308}(I)=\frac{120298}{10!}=0.033$,&  $P_{5,309}(I)=\frac{200223}{10!}=0.055$,& \\[5pt]
        &$P_{5,1058}(I)=\frac{357868}{10!}=0.0986$,&$P_{5,1059}(I)=\frac{139821}{10!}=0.0385$&\\[5pt]\hline
    \end{tabular}
    \vspace{2mm}
    \caption{Proportion of $e$-power factoradic $p$-happy numbers of $\Se$ in the interval $I=[1,10!]$, for $1\leq e\leq 5$. }
    \label{tab:proportion}
\end{table}

We consider $e\in\{1,2,3,4,5\}$ and let $\ell\in\Z^+$.
If $I=[1,\ell]$ is an interval of $\Z_{\geq 0}$ and $p$ is an $e$-power factoradic fixed point of $\Se$, then the proportion of $e$-power factoradic $p$-happy numbers of $\Se$ in the interval $I$ is given by   
\[P_{e,p}(I)=\frac{|\{n\in I:S_{e,!}^k(n)=p\mbox{ for some $k\in \Z_{\geq 0}$}\}|}{|I|}.\]
In Table \ref{tab:proportion} we present values for $P_{e,p}(I)$ for $e\in\{1,2,3,4\}$ and $I=[1,10!]$. In light of these computations, we pose the following open problem.
\begin{problem}\label{prob:1}
For $e\geq 2$, give lower and upper bounds for the density of $e$-power factoradic $p$-happy numbers.
\end{problem}

\noindent MSC2010: 11A63
\addresseshere

\end{document}